\documentclass[oneside, reqno] {amsart}

\usepackage{xypic}
\input xy
\xyoption{all}
\usepackage{epsfig}
\usepackage{amsthm}
\usepackage{amssymb}
\usepackage{amsmath}
\usepackage{amscd}
\usepackage{color}
\usepackage[T1]{fontenc}
\usepackage{upgreek}
\usepackage[font=scriptsize]{caption}
\usepackage{wrapfig}

\usepackage{graphicx}
\usepackage{subfigure}

\newcommand{\bg}{\begin{equation}}
\newcommand{\ed}{\end{equation}}
\newcommand{\bga}{\begin{eqnarray}}
\newcommand{\eda}{\end{eqnarray}}

\def\cbdu{\par{\raggedleft$\Box$\par}}

\newtheorem {Theorem}  {Theorem}

\numberwithin{Theorem}{section}

\newtheorem {Lemma}[Theorem]  {Lemma}
\newtheorem {Proposition}[Theorem]{Proposition}
\theoremstyle{definition}

\theoremstyle{remark}

\expandafter\chardef\csname pre amssym.def
at\endcsname=\the\catcode`\@ \catcode`\@=11
\def\undefine#1{\let#1\undefined}
\def\newsymbol#1#2#3#4#5{\let\next@\relax
 \ifnum#2=\@ne\let\next@\msafam@\else
 \ifnum#2=\tw@\let\next@\msbfam@\fi\fi
 \mathchardef#1="#3\next@#4#5}
\def\mathhexbox@#1#2#3{\relax
 \ifmmode\mathpalette{}{\m@th\mathchar"#1#2#3}%
 \else\leavevmode\hbox{$\m@th\mathchar"#1#2#3$}\fi}
\def\hexnumber@#1{\ifcase#1 0\or 1\or 2\or 3\or 4\or 5\or 6\or 7\or 8\or
 9\or A\or B\or C\or D\or E\or F\fi}

\font\teneufm=eufm10 \font\seveneufm=eufm7 \font\fiveeufm=eufm5
\newfam\eufmfam
\textfont\eufmfam=\teneufm \scriptfont\eufmfam=\seveneufm
\scriptscriptfont\eufmfam=\fiveeufm

\catcode`\@=\csname pre amssym.def at\endcsname

\newcounter{remark}
\renewcommand{\k}{\kappa}

\renewcommand{\div}{\mbox{div}}

\def  \12  {{\frac{1}{2}}}

\def\build#1_#2^#3{\mathrel{\mathop{\kern 0pt#1}\limits_{#2}^{#3}}}

\numberwithin{equation}{section}

\begin{document}

\title[EMHD DSS solutions]{Discretely self-similar solutions for electron MHD}


\author [Nada Adzic Vukotic]{Nada Adzic Vukotic}

\address{Department of Mathematics, Statistics and Computer Science, University of Illinois at Chicago, Chicago, IL 60607, USA}
\email{nadzic2@uic.edu} 

\author [Mimi Dai]{Mimi Dai}

\address{Department of Mathematics, Statistics and Computer Science, University of Illinois at Chicago, Chicago, IL 60607, USA}
\email{mdai@uic.edu}

\thanks{The authors are partially supported by the NSF grant DMS--2308208 and Simons Foundation. }

\begin{abstract}
We study discretely self-similar solutions for the electron magnetohydrodynamics (MHD) without resistivity. Under several different decay and non-decay conditions, we show the absence of non-trivial discretely self-similar blowup solutions. 

\bigskip

KEY WORDS: magnetohydrodynamics; Hall effect; discretely self-similar solutions; singularity.

\hspace{0.02cm}CLASSIFICATION CODE: 35Q35, 76B03, 76D09, 76E25, 76W05.
\end{abstract}

\maketitle

\section{Introduction}

\medskip

\subsection{Overview}

The non-resistive electron magnetohydrodynamics (MHD) 
\begin{equation}\label{emhd}
\begin{split}
B_t+ \nabla\times ((\nabla\times B)\times B)&=0,\\
\nabla\cdot B&= 0
\end{split}
\end{equation}
is an important model in plasma physics to capture the rapid magnetic reconnection phenomena due to the Hall effect. It is obtained from the full non-resistive MHD system with Hall effect when the ion motion is assumed to be stationary and only the electron motion is considered (cf. \cite{ADFL, Bis1}). The unknown vector $B$ is the magnetic field which satisfies the Gauss law for magnetism, the second equation of \eqref{emhd}.  The nonlinear term in (\ref{emhd}) derived from the Hall effect is referred as the Hall term (cf. \cite{BDS}), which is rather singular compared to the nonlinear term of the well-known Euler equation. Understanding the dynamics of the nonlinear Hall term is a great challenge. In particular, when resistivity is absent, that is there is no dissipation term $\Delta B$ in the equation, the solution of \eqref{emhd} (if exists) is expected to be singular. In this paper, we study a particular type of singular solutions for \eqref{emhd},  the discretely self-similar solutions. 


To introduce self-similar solutions for the electron MHD, we start by considering the scaling of the equation. 
Assume $B(x,t)$ is a solution of \eqref{emhd} with the initial data $B_0(x,t)$. One can check that, for any parameter $\lambda\in \mathbb R$,  the rescaled function $B_\lambda$ defined as
\begin{equation}\label{scale}
B_\lambda(x,t)=\lambda^\alpha B(\lambda x, \lambda^{\alpha+2}t), \ \ \forall \ \ \alpha\in \mathbb R
\end{equation}
is a solution to \eqref{emhd} as well, but with the rescaled initial data $\lambda^\alpha B_0(\lambda x)$. 

If the rescaled function $B_\lambda$ in \eqref{scale} is a solution for any $\lambda>1$, it is referred a self-similar solution; if the solution satisfies the rescaling for only one $\lambda>1$, it is called a discretely self-similar solution. 


In light of the scaling \eqref{scale}, we can consider solutions to \eqref{emhd} in the self-similar form with respect to the potential singularity point $(0,0)$
\begin{equation}\label{self-blowup}
B(x,t)=\frac{1}{(-t)^{\frac{\alpha}{\alpha+2}}}H\left( \frac{x}{(-t)^{\frac{1}{\alpha+2}}}, s\right), \ \ \alpha\neq -2
\end{equation}
with $s=-\log (-t)$. Denote $y=\frac{x}{(-t)^{\frac{1}{\alpha+2}}}$. As $t\to0$, the self-similar variables $s\to \infty$ and $|y|\to \infty$.
The profile vector field $H(y,s)$ satisfies the system
\begin{equation}\label{eq-H}
\begin{split}
\partial_s H+\frac{\alpha}{\alpha+2}H+\frac{1}{\alpha+2} (y\cdot \nabla) H+\nabla\times ((\nabla\times H)\times H)=&\ 0,\\
\nabla\cdot H=&\ 0.
\end{split}
\end{equation}
The existence of a non-trivial solution $H$ of \eqref{eq-H} corresponds to the existence of a solution $B$ of \eqref{emhd} in the form \eqref{self-blowup} that blows up at time $0$.  The solution $B$ is self-similar if and only if $H$ is independent $s$. While $B$ is discretely self-similar with some $\lambda>1$ if and only if $H$ is periodic in $s$, that is, 
\[H(y,s)=H(y, s+S_0), \ \ \forall \ (y,s)\in \mathbb R^{3+1}\]
with $S_0=(\alpha+2)\log \lambda$. Indeed, if $B$ is discretely self-similar with $\lambda>1$, it follows from \eqref{scale} that 
\[B_{\lambda^n}(x,t)=\lambda^{n\alpha} B(\lambda^n x, \lambda^{n(\alpha+2)}t)\]
for any integer $n$ is a solution. Letting $\lambda^{n(\alpha+2)}(-t)=1$, we have $-t=\lambda^{-n(\alpha+2)}$ and hence
\[s=-\log(-t)= -\log(\lambda^{-n(\alpha+2)})=n(\alpha+2)\log\lambda\]
which indicates $S_0=(\alpha+2)\log\lambda$ is the time period. 

Self-similar and discretely self-similar solutions have been investigated for hydrodynamic equations, for instance for Euler equations in \cite{Chae11, Chae15, CT}.
Self-similar solutions for the electron MHD \eqref{emhd} were studied in \cite{Dai-GW}, where the authors obtained several criteria that exclude the existence of non-trivial self-similar blowup solutions. As a continuation, we focus on discretely self-similar solutions for the electron MHD in this article and study the general equation \eqref{eq-H}. The goal is to identity conditions, with both spatial decay and non-decay types, which can rule out the existence of 
discretely self-similar blowup solutions.

\medskip

\subsection{Main results}

Denote the current profile $J=\nabla\times H$.
Depending on the different values of the scaling parameter $\alpha$, we first state a few results where the conditions on the profile $H$ and $J$ indicate certain decay property.

\begin{Theorem}\label{thm-1}
Let $\alpha\neq -2$ and $\nu=0$. Assume  
\begin{itemize}
\item [(i)] $H\in C_y^2C_s^1(\mathbb R^{3+1})$,
\item [(ii)] $J(y)=o(|y|)$ as $|y|\to\infty$,
\item [(iii)] for some $r>0$, 
\[H\in \cap_{0<q<r} L^q(\mathbb R^3\times [0,S_0]),\]
\end{itemize}
then $H\equiv 0$ on $\mathbb R^{3+1}$. 
\end{Theorem}

\medskip

\begin{Theorem}\label{thm-2}
Let $p>\frac65$. 
Assume either
\begin{itemize}
\item [(i)] $\alpha>\frac32$ and $H\in L^3_s W_y^{1,p}$, or
\item [(ii)] $-2<\alpha<\frac32$ and $H\in L^2_s L^2_y\cap L^3_sW_y^{1,p}$,
\end{itemize}
then $H\equiv 0$ on $\mathbb R^{3+1}$. 
\end{Theorem}

\begin{Theorem}\label{thm-3}
Let $H\in C_y^2C_s^1(\mathbb R^{3+1})$ be a time periodic solution of \eqref{eq-H} with period $S_0>0$ satisfying 
\begin{equation}\label{decay}
\lim_{|y|\to\infty} \sup_{0<s<S_0}\left(|J(y,s)|+ |\nabla J(y,s)|\right)=0.
\end{equation}
In addition, we assume 
either
\begin{itemize}
\item [(i)] $\alpha>0$ and there exists $0<q<\frac3\alpha$, or
\item [(ii)] $-2<\alpha\leq 0$ and there exists $q>0$,
\end{itemize}
such that $H\in L^q(\mathbb R^3\times [0,S_0])$.
 Then $H\equiv 0$ on $\mathbb R^{3+1}$.
\end{Theorem}

\begin{Theorem}\label{thm-4}
Let $H\in C_y^2C_s^1(\mathbb R^{3+1})$ be a time periodic solution of \eqref{eq-H} with period $S_0>0$ and $\alpha<-2$ satisfying 
\begin{equation}\label{decay4}
H(y,s)=o(1), \ \quad \nabla J(y,s)=o(1) \ \ \mbox{as} \ \ |y|\to \infty \ \ \forall \ \ s\in[0,S_0]. 
\end{equation}
 Then $H\equiv 0$ on $\mathbb R^{3+1}$.
\end{Theorem}

\begin{Theorem}\label{thm-5}
Let $H\in C_y^3C_s^1(\mathbb R^{3+1})$ be a time periodic solution of \eqref{eq-H} with period $S_0>0$ and $\alpha>0$. Assume $y=0$ is a local extremal point of $J_i(\cdot, s)$ for all $i=1,2,3$ and $J(0,s)=0$ for all $s\in(0,S_0)$.
Then $H\equiv 0$ on $\mathbb R^{3+1}$.
\end{Theorem}

Next we show that, for $\alpha>\frac32$, a non-decay condition on $H$ can also exclude the existence of discretely self-similar solution.

\begin{Theorem}\label{thm-6}
For $\alpha>\frac32$, let $H\in C^\infty(\mathbb R^3\times \mathbb R)$ be a time periodic solution to \eqref{eq-H} with period $S_0>0$. If 
\begin{equation}\label{non-decay}
\sup_{s\in \mathbb R} |H(y,s)|=o(|y|^2) \ \ \ \mbox{as} \ \ |y|\to \infty,
\end{equation}
we have $H\equiv 0$ on $\mathbb R^{3+1}$.
\end{Theorem}



\medskip

\subsection{Notations}
We denote $C$ or $c$ by a generic constant which may be different from line to line. The symbol $\lesssim$ is used as an inequality up to a non-significant constant.

We denote $B_r(x_0)=\{y\in\mathbb R^3: |y-x_0|<r\}$ the ball centered at $x_0$ with radius $r>0$. If $x_0$ is the origin, we simply denote $B_r=B_r(0)$.

In the rest of the paper, each section is devoted to the proof of a theorem stated above.

\bigskip

\section{Proof of Theorem \ref{thm-1}}
\label{sec-thm1}

Define the smooth radial cut-off function $\varphi\in C^\infty_0(\mathbb R)$ as
\begin{equation}\notag
\varphi(|x|)=
\begin{cases}
1, \ \ \ |x|<1\\
0, \ \ \ |x|>2
\end{cases}
\end{equation}
satisfying $0\leq \varphi(|x|)\leq 1$ for $1\leq |x|\leq 2$. We then define the rescaled cut-off function 
\[\varphi_R(y)=\varphi(|y|/R).\]

Denote $a=\frac{\alpha}{\alpha+2}$ and $b=\frac{1}{\alpha+2}$. Rewriting the Hall term 
\[\nabla\times((\nabla\times H)\times H)=(H\cdot\nabla) J-(J\cdot \nabla) H\]
 in the profile equation \eqref{eq-H}, we arrive at 
 \begin{equation}\label{eq-H2}
\begin{split}
\partial_s H+aH+b(y\cdot \nabla) H+(H\cdot\nabla) J-(J\cdot \nabla) H &= 0,\\
\nabla\cdot H &= 0.
\end{split}
\end{equation}
For $q>0$, taking inner product of the first equation of \eqref{eq-H2} with $|H|^{q-2}H\varphi_R$, integrating over $\mathbb R^3\times (0,S_0)$ and applying integration by parts yields
\begin{equation}\notag
\begin{split}
&\int_0^{S_0}\int_{\mathbb R^3}\frac1q\partial_s |H|^q\varphi_R dyds+(a-\frac{3b}q) \int_0^{S_0}\int_{\mathbb R^3} |H|^q\varphi_R dyds\\
&-\frac{b}q\int_0^{S_0}\int_{\mathbb R^3} (y\cdot\nabla \varphi_R)|H|^q dyds+\frac1q\int_0^{S_0}\int_{\mathbb R^3} (J\cdot\nabla) \varphi_R|H|^q dyds\\
&+\int_0^{S_0}\int_{\mathbb R^3} (\frac{H}{|H|}\cdot\nabla )J\cdot \frac{H}{|H|} |H|^q  \varphi_Rdyds=0.
\end{split}
\end{equation}
Applying the time periodicity property of $H$, the first integral in the equation above vanishes. Hence we end up with the energy identity 
\begin{equation}\label{energy-1}
\begin{split}
&(a-\frac{3b}q) \int_0^{S_0}\int_{\mathbb R^3} |H|^q\varphi_R dyds
+\int_0^{S_0}\int_{\mathbb R^3} (\frac{H}{|H|}\cdot\nabla )J\cdot \frac{H}{|H|} |H|^q  \varphi_Rdyds\\
=&\frac{1}q\int_0^{S_0}\int_{\mathbb R^3} (by\cdot\nabla \varphi_R)|H|^q dyds-\frac{1}q\int_0^{S_0}\int_{\mathbb R^3} (J\cdot\nabla \varphi_R)|H|^q dyds.
\end{split}
\end{equation}
Note 
\begin{equation}\notag
\nabla \varphi_R
\begin{cases}
=0, \ \ \ |y|<R \ \ \mbox{or} \ \ |y|>2R\\
\leq \frac{\|\varphi'\|_{L^\infty}}{R}, \ \ \ R\leq |y|\leq 2R.
\end{cases}
\end{equation}
Thus the right hand side of \eqref{energy-1} has the bound
\[
\begin{split}
&\left| \frac{1}q\int_0^{S_0}\int_{\mathbb R^3} (by\cdot\nabla \varphi_R)|H|^q dyds-\frac{1}q\int_0^{S_0}\int_{\mathbb R^3} (J\cdot\nabla \varphi_R)|H|^q dyds\right|\\
\leq &\ \frac{\|\varphi'\|_{L^\infty}}{q} \int_0^{S_0}\int_{R\leq |y|\leq 2R} |H|^q\frac{|y|+|J|}{R} dyds\\
\leq &\ C  \int_0^{S_0}\int_{R\leq |y|\leq 2R} |H|^q dyds
\end{split}
\]
where in the last step we used the assumption (ii) of the theorem. 
On the other hand, since $|(\frac{H}{|H|}\cdot\nabla )J\cdot \frac{H}{|H|}|\leq |\nabla J|$, we have 
\[
\begin{split}
&\left| \int_0^{S_0}\int_{\mathbb R^3} (\frac{H}{|H|}\cdot\nabla )J\cdot \frac{H}{|H|} |H|^q  \varphi_Rdyds \right|\\
\leq &  \int_0^{S_0}\int_{\mathbb R^3} |\nabla J| |H|^q  \varphi_Rdyds \\
\leq &\ C  \int_0^{S_0}\int_{R\leq |y|\leq 2R} |H|^q dyds.
\end{split}
\]
Combining the previous two estimates with \eqref{energy-1} gives
\begin{equation}\notag
\begin{split}
|a-\frac{3b}q| \int_0^{S_0}\int_{\mathbb R^3} |H|^q\varphi_R dyds\leq C  \int_0^{S_0}\int_{R\leq |y|\leq 2R} |H|^q dyds.
\end{split}
\end{equation}
Taking $R\to\infty$, we have 
\[ \int_0^{S_0}\int_{R\leq |y|\leq 2R} |H|^q dyds \to 0\]
thanks to the assumption (iii). While $|a-\frac{3b}q| \to \infty$ as $q\to 0$. Hence we obtain 
\[\int_0^{S_0}\int_{\mathbb R^3} |H|^q dyds=0\]
which implies $H\equiv 0$ on $\mathbb R^{3+1}$.

\bigskip

\section{Proof of Theorem \ref{thm-2}}
\label{sec-thm2}

Let $\varphi_R$ be the radial cut-off function defined in Section \ref{sec-thm1}, with the additional property that it is non-increasing on $[0,\infty)$. 
We take inner product of the first equation of \eqref{eq-H} with $H\varphi_R$ and integrate over $\mathbb R^3\times [0,S_0]$ to obtain 
\begin{equation}\notag
\begin{split}
&\int_0^{S_0}\int_{\mathbb R^3}\frac12\partial_s |H|^2\varphi_R dyds+a \int_0^{S_0}\int_{\mathbb R^3} |H|^2\varphi_R dyds
+b\int_0^{S_0}\int_{\mathbb R^3} (y\cdot\nabla) H\cdot H\varphi_R dyds\\
&+\int_0^{S_0}\int_{\mathbb R^3} \nabla\times((\nabla\times H)\times H)\cdot (H\varphi_R) dyds=0.
\end{split}
\end{equation}
Again the first integral vanishes due the time periodicity in $s$. Applying integration by parts, the third integral can be written as
\begin{equation}\notag
\begin{split}
&\quad b\int_0^{S_0}\int_{\mathbb R^3} (y\cdot\nabla) H\cdot H\varphi_R dyds\\
&=\frac{b}2\int_0^{S_0}\int_{\mathbb R^3} (y\cdot\nabla) |H|^2\varphi_R dyds\\
&=\frac{b}2\int_0^{S_0}\int_{\mathbb R^3} \div (y |H|^2\varphi_R) dyds-\frac{b}2\int_0^{S_0}\int_{\mathbb R^3} \div (y) |H|^2\varphi_R dyds\\
&\quad- \frac{b}2\int_0^{S_0}\int_{\mathbb R^3} (y\cdot\nabla \varphi_R) |H|^2 dyds\\
&=-\frac{3b}2\int_0^{S_0}\int_{\mathbb R^3}  |H|^2\varphi_R dyds- \frac{b}2\int_0^{S_0}\int_{\mathbb R^3} (y\cdot\nabla \varphi_R) |H|^2 dyds.
\end{split}
\end{equation}
While the last integral becomes after integration by parts
\begin{equation}\notag
\begin{split}
&\quad \int_0^{S_0}\int_{\mathbb R^3} \nabla\times((\nabla\times H)\times H)\cdot (H\varphi_R) dyds\\
&=\int_0^{S_0}\int_{\mathbb R^3} ((\nabla\times H)\times H)\cdot \nabla\times (H\varphi_R) dyds\\
&=\int_0^{S_0}\int_{\mathbb R^3} ((\nabla\times H)\times H)\cdot (\nabla\times H)\varphi_R dyds\\
&\quad -\int_0^{S_0}\int_{\mathbb R^3} ((\nabla\times H)\times H)\cdot (H\times \nabla \varphi_R) dyds\\
&=-\int_0^{S_0}\int_{\mathbb R^3} ((\nabla\times H)\times H)\cdot (H\times \nabla \varphi_R) dyds.
\end{split}
\end{equation}
Combining the analysis above gives the energy identity
\begin{equation}\label{energy-2}
\begin{split}
&\quad (a-\frac{3b}2) \int_0^{S_0}\int_{\mathbb R^3} |H|^2\varphi_R dyds- \frac{b}2\int_0^{S_0}\int_{\mathbb R^3} (y\cdot\nabla \varphi_R) |H|^2 dyds\\
&=\int_0^{S_0}\int_{\mathbb R^3} ((\nabla\times H)\times H)\cdot (H\times \nabla \varphi_R) dyds.
\end{split}
\end{equation}
Recall $a=\frac{\alpha}{\alpha+2}$ and $b=\frac{1}{\alpha+2}$. For $\alpha>-2$, b>0. Note that
\[y\cdot \nabla\varphi_R=\sum_{i=1}^3y_i\partial_i \varphi_R=\sum_{i=1}^3\frac{y_i^2}{R|y|}\varphi'(\frac{|y|}R)=\frac{|y|}{R}\varphi'(\frac{|y|}{R})\leq 0\]
since $\varphi$ is non-increasing. Hence it follows from \eqref{energy-2} that
\begin{equation}\label{energy-3}
(a-\frac{3b}2) \int_0^{S_0}\int_{\mathbb R^3} |H|^2\varphi_R dyds
\leq \int_0^{S_0}\int_{\mathbb R^3} ((\nabla\times H)\times H)\cdot (H\times \nabla \varphi_R) dyds.
\end{equation}

{\textbf{ Case(i).}}
Since $a-\frac{3b}2=\frac{2\alpha-3}{2(\alpha+2)}$, if $\alpha>\frac32$, we have $a-\frac{3b}2>0$. Thus \eqref{energy-3} implies
\begin{equation}\label{energy-4}
\begin{split}
&\quad (a-\frac{3b}2) \int_0^{S_0}\int_{\mathbb R^3} |H|^2\varphi_{R_1} dyds\\
&\leq (a-\frac{3b}2) \int_0^{S_0}\int_{\mathbb R^3} |H|^2\varphi_{R_2} dyds\\
&\leq \int_0^{S_0}\int_{\mathbb R^3} ((\nabla\times H)\times H)\cdot (H\times \nabla \varphi_{R_2}) dyds
\end{split}
\end{equation}
for any $0<R_1<R_2<\infty$. 

For $1<p, p_1, p_2<\infty$ satisfying $\frac1p+\frac1{p_1}+\frac1{p_2}=1$, applying H\"older's inequality to the last integral in \eqref{energy-4} gives
\begin{equation}\notag
\begin{split}
&\quad\int_{\mathbb R^3} ((\nabla\times H)\times H)\cdot (H\times \nabla \varphi_{R_2}) dyds\\
&\leq \frac{\|\varphi'\|_{L^\infty}}{R_2} \int_{R_2<|y|<2R_2} |\nabla H||H|^2\, dyds\\
&\lesssim \frac{\|\varphi'\|_{L^\infty}}{R_2} \|\nabla H\|_{L^p(B_{2R_2}\backslash B_{R_2})}\|H\|_{L^{2p_1}(B_{2R_2}\backslash B_{R_2})}^2 R_2^{\frac3{p_2}}\\
&\lesssim \|\varphi'\|_{L^\infty} \|\nabla H\|_{L^p(B_{2R_2}\backslash B_{R_2})}^3 R_2^{\frac3{p_2}-1}\\
\end{split}
\end{equation}
where we used Sobolev imbedding provided $\frac{6-2p}{3p}\leq\frac{1}{p_1}<1$, which is satisfied for $p>\frac65$. Putting this estimate together with \eqref{energy-4} leads to 
\begin{equation}\label{energy-5}
\begin{split}
(a-\frac{3b}2) \int_0^{S_0}\int_{\mathbb R^3} |H|^2\varphi_{R_1} dyds
\lesssim \|\varphi'\|_{L^\infty}  \int_0^{S_0}\|\nabla H\|_{L^p(B_{2R_2}\backslash B_{R_2})}^3 \, dsR_2^{\frac3{p_2}-1}.
\end{split}
\end{equation}
In view of the assumption (i) of Theorem \ref{thm-2}, if we take $p_2>1$ such that $\frac3{p_2}-1\leq 0$, we have
\[\lim_{R_2\to \infty} \int_0^{S_0}\|\nabla H\|_{L^p(B_{2R_2}\textbackslash B_{R_2})}^3 \, dsR_2^{\frac3{p_2}-1}=0 .\]
We remark that, for any $p>\frac65$, it is possible to find $p_2, p_3>1$ satisfying $\frac1p+\frac1{p_1}+\frac1{p_2}=1$ and $\frac3{p_2}-1\leq 0$. 
Since $a-\frac{3b}2>0$, it follows from \eqref{energy-5} that
\[ \int_0^{S_0}\int_{\mathbb R^3} |H|^2\varphi_{R_1} dyds=0, \ \ \ \forall \ \ R_1>0.\]
Therefore we have $H\equiv 0$ on $\mathbb R^{3+1}$.

{\textbf{ Case(ii).}} If $-2<\alpha<\frac32$, it follows from \eqref{energy-2} that
\begin{equation}\notag
\begin{split}
&\quad |a-\frac{3b}2| \int_0^{S_0}\int_{\mathbb R^3} |H|^2\varphi_R dyds\\
&\leq \frac{b}2\int_0^{S_0}\int_{\mathbb R^3} |y\cdot\nabla \varphi_R| |H|^2 dyds
+\int_0^{S_0}\int_{\mathbb R^3} |(\nabla\times H)\times H| |H\times \nabla \varphi_R| dyds.
\end{split}
\end{equation}
Again, as in Case (i), for $H\in L^3_s W_y^{1,p}$ with $p>\frac65$, we can show that 
\[\lim_{R\to \infty} \int_0^{S_0}\int_{\mathbb R^3} |(\nabla\times H)\times H| |H\times \nabla \varphi_R| dyds =0.\]
On the other hand, we deduce 
\begin{equation}\notag
\begin{split}
\int_0^{S_0}\int_{\mathbb R^3} |y\cdot\nabla \varphi_R| |H|^2 dyds&\leq \int_0^{S_0}\int_{R<|y|<2R} \frac{|y||\varphi'|}R |H|^2 dyds\\
&\lesssim  \int_0^{S_0}\int_{R<|y|<2R}  |H|^2 dyds\\
&\to 0
\end{split}
\end{equation}
as $R\to\infty$, since $H\in L^2_s L^2_y$. Therefore, we again obtain 
\[\int_0^{S_0}\int_{\mathbb R^3} |H|^2 dyds=0\]
and hence $H\equiv 0$ on $\mathbb R^{3+1}$.

\bigskip

\section{Proof of Theorem \ref{thm-3}}
\label{sec-thm3}

Taking dot product of the first equation in \eqref{eq-H2} with $|H|^{q-2}H$ yields
\begin{equation}\notag
\begin{split}
&\frac1q \partial_s |H|^q+a|H|^q+b(y\cdot\nabla)H\cdot |H|^{q-2}H\\
&+(H\cdot\nabla) J\cdot |H|^{q-2}H-(J\cdot\nabla )H\cdot |H|^{q-2}H=0.
\end{split}
\end{equation}
Note we can write
\begin{equation}\notag
\begin{split}
(y\cdot\nabla)H\cdot |H|^{q-2}H&=(y\cdot\nabla)\frac{|H|^q}{q}
=\div (y\frac{|H|^q}{q})-(\div y) \frac{|H|^q}{q}=\div (y\frac{|H|^q}{q})-3 \frac{|H|^q}{q}
\end{split}
\end{equation}
and 
\[(J\cdot\nabla )H\cdot |H|^{q-2}H=(J\cdot\nabla) \frac{|H|^q}{q}=\div (J \frac{|H|^q}{q}) \]
since $\div J=0$. Combining the last three identities gives 
\begin{equation}\notag
\begin{split}
&\frac1q \partial_s |H|^q+(a-\frac{3b}q)|H|^q+\frac{b}q \div (y |H|^q)\\
&+(\frac{H}{|H|}\cdot\nabla) J\cdot \frac{H}{|H|} |H|^{q}-\frac1q \div (J |H|^q)=0.
\end{split}
\end{equation}
For any $0<R_1<R_2$, we have by integrating the previous equation over $(B_{R_2}\backslash B_{R_1})\times (0,S_0)$ 
\begin{equation}\label{energy-31}
\begin{split}
&\frac1q \int_0^{S_0}\int_{R_1<|y|<R_2} \partial_s |H|^q\, dyds+(a-\frac{3b}q)\int_0^{S_0}\int_{R_1<|y|<R_2} |H|^q\, dyds\\
&+\frac{b}q\int_0^{S_0}\int_{R_1<|y|<R_2}  \div (y |H|^q)\, dyds
+\int_0^{S_0}\int_{R_1<|y|<R_2} (\frac{H}{|H|}\cdot\nabla) J\cdot \frac{H}{|H|} |H|^{q}\, dyds\\
&-\frac1q\int_0^{S_0}\int_{R_1<|y|<R_2}  \div (J |H|^q)\, dyds=0.
\end{split}
\end{equation}
In light of the time periodicity of $H$ in $s$, it is clear that 
\[ \int_0^{S_0}\int_{R_1<|y|<R_2} \partial_s |H|^q\, dyds=\int_{R_1<|y|<R_2} |H(y,S_0)|^q\, dy-\int_{R_1<|y|<R_2} |H(y,0)|^q\, dy=0.\]
Applying divergence theorem, we obtain
\begin{equation}\notag
\begin{split}
&\quad \int_0^{S_0}\int_{R_1<|y|<R_2}  \div (y |H|^q)\, dyds\\
&=\int_0^{S_0}\int_{|y|=R_2} y\cdot \frac{y}{|y|} |H|^q\, dyds-\int_0^{S_0}\int_{|y|=R_1} y\cdot \frac{y}{|y|} |H|^q\, dyds\\
&=R_2 \int_0^{S_0}\int_{|y|=R_2}  |H|^q\, dyds-R_1 \int_0^{S_0}\int_{|y|=R_1}  |H|^q\, dyds
\end{split}
\end{equation}
and 
\begin{equation}\notag
\begin{split}
&\quad \int_0^{S_0}\int_{R_1<|y|<R_2}  \div (J |H|^q)\, dyds\\
&=\int_0^{S_0}\int_{|y|=R_2} J\cdot \frac{y}{|y|} |H|^q\, dyds-\int_0^{S_0}\int_{|y|=R_1} J\cdot \frac{y}{|y|} |H|^q\, dyds.\\
\end{split}
\end{equation}
Putting the last equations into \eqref{energy-31} yields
\begin{equation}\label{energy-32}
\begin{split}
&(a-\frac{3b}q)\int_0^{S_0}\int_{R_1<|y|<R_2} |H|^q\, dyds
+\int_0^{S_0}\int_{R_1<|y|<R_2} (\frac{H}{|H|}\cdot\nabla) J\cdot \frac{H}{|H|} |H|^{q}\, dyds\\
&+\frac{b}q R_2 \int_0^{S_0}\int_{|y|=R_2}  |H|^q\, dSds-\frac{b}q R_1\int_0^{S_0}\int_{|y|=R_1}  |H|^q\, dSds\\
&-\frac1{q} \int_0^{S_0}\int_{|y|=R_2} J\cdot \frac{y}{|y|} |H|^q\, dSds+\frac1{q} \int_0^{S_0}\int_{|y|=R_1} J\cdot \frac{y}{|y|} |H|^q\, dSds=0.
\end{split}
\end{equation}
Since $H\in L^q(\mathbb R^3\times [0,S_0])$, we have 
\[\int_0^{S_0}\int_{\mathbb R^3}|H|^q\, dyds=\int_0^{\infty}\int_0^{S_0} \int_{|y|=r}|H|^q\, dS_r ds dr<\infty, \]
which implies (at least for a sequence $\{R_{2,j}\}_{j=1}^\infty$)
\[\lim_{R_2\to\infty}R_2 \int_0^{S_0}\int_{|y|=R_2}  |H|^q\, dSds =0.\]
Additionally, it follows from assumption \eqref{decay} that
\[\lim_{R_2\to\infty}\int_0^{S_0}\int_{|y|=R_2} J\cdot \frac{y}{|y|} |H|^q\, dSds =0.\]
Hence, taking $R_2\to\infty$ in \eqref{energy-32} gives
\begin{equation}\label{energy-33}
\begin{split}
&\quad(a-\frac{3b}q)\int_0^{S_0}\int_{|y|>R_1} |H|^q\, dyds-\frac{b}q R_1\int_0^{S_0}\int_{|y|=R_1}  |H|^q\, dSds\\
&=-\int_0^{S_0}\int_{|y|>R_1} (\frac{H}{|H|}\cdot\nabla) J\cdot \frac{H}{|H|} |H|^{q}\, dyds-\frac1{q} \int_0^{S_0}\int_{|y|=R_1} J\cdot \frac{y}{|y|} |H|^q\, dSds.
\end{split}
\end{equation}
In view of the decay assumption \eqref{decay} again, we have for large enough $R_1>0$ that
\[\left|(\frac{H}{|H|}\cdot\nabla) J\cdot \frac{H}{|H|}\right|\leq |\nabla J|\leq \frac12|a-\frac{3b}q|,\]
\[\left| J\cdot\frac{y}{|y|}\right|\leq \frac12\frac{|b|R_1}{q}.\]
Now we assume $a-\frac{3b}q>0$ and $-b>0$.
It follows from \eqref{energy-33} that
\begin{equation}\label{conclusion}
\int_0^{S_0}\int_{|y|>R_1} |H|^q\, dyds=\int_0^{S_0}\int_{|y|=R_1}  |H|^q\, dSds=0.
\end{equation}
On the other hand, if $a-\frac{3b}q<0$ and $-b<0$, multiplying $(-1)$ on each side of equation \eqref{energy-33} and the conclusion \eqref{conclusion} follows as well.
Recalling the definition of $a$ and $b$, in order to guarantee $(a-\frac{3b}q)(-b)>0$, we need to have that either $\{-2<\alpha\leq 0\}$ and any $q>0$, or $\alpha>0$ and  $0<q<\frac3\alpha$.

In the end, the property in \eqref{conclusion} implies 
\[H\in \cap_{0<p<q} L^p(\mathbb R^3\times [0,S_0]).\]
Thus the conclusion $H\equiv 0$ follows from Theorem \ref{thm-1}.

\bigskip

\section{Proof of Theorem \ref{thm-4}}
\label{sec-thm4}

In this section, we define the smooth radial cut-off function $\varphi\in C^\infty_0(\mathbb R)$ as
\begin{equation}\notag
\varphi(|x|)=
\begin{cases}
0, \ \ \ |x|<1\\
1, \ \ \ |x|>2
\end{cases}
\end{equation}
satisfying $0\leq \varphi(|x|)\leq 1$ for $1\leq |x|\leq 2$ and $\varphi'\geq 0$. Denote the rescaled non-decreasing cut-off function 
\[\varphi_R(y)=\varphi(|y|/R).\]

In order to prove the statement of Theorem \ref{thm-4}, we first show that the profile $H$ has compact support in $\mathbb R^3\times [0,S_0]$ under the decay assumption of the theorem. We achieve this applying a contradiction argument. As a consequence, we justify the assumption (iii) of Theorem \ref{thm-1} as well as (ii) to conclude the proof.

As before, we start by taking dot product of equation \eqref{eq-H2} with $H\varphi_R$ to obtain 
\begin{equation}\notag
\begin{split}
&\partial_s(\frac12|H|^2\varphi_R)+a|H|^2\varphi_R+b(y\cdot\nabla)H\cdot H\varphi_R\\
&+(H\cdot\nabla)J\cdot H\varphi_R-(J\cdot\nabla)H\cdot H\varphi_R=0.
\end{split}
\end{equation}
Since 
\[b(y\cdot\nabla)H\cdot H\varphi_R=b(y\cdot\nabla)(\frac12|H|^2\varphi_R)-b(y\cdot\nabla)\varphi_R \frac12|H|^2,\]
\[(J\cdot\nabla)H\cdot H\varphi_R=(J\cdot\nabla)(\frac12|H|^2\varphi_R)-(J\cdot\nabla)\varphi_R \frac12|H|^2,\]
it follows
\begin{equation}\label{energy-41}
\begin{split}
&\quad \partial_s(|H|^2\varphi_R)+2a|H|^2\varphi_R+b(y\cdot\nabla)(|H|^2\varphi_R)-(J\cdot\nabla)(|H|^2\varphi_R)\\
&=-2(\frac{H}{|H|}\cdot\nabla)J\cdot \frac{H}{|H|} |H|^2\varphi_R+b(y\cdot\nabla)\varphi_R|H|^2-(J\cdot\nabla)\varphi_R |H|^2.
\end{split}
\end{equation}
We analyze the terms on the right hand side of \eqref{energy-41} in the following. First of all, in view of the assumption \eqref{decay4}, we have
\begin{equation}\label{far-1}
|(\frac{H}{|H|}\cdot\nabla)J\cdot \frac{H}{|H|}| |H|^2\varphi_R\leq |\nabla J(y,s)| |H|^2\varphi_R\leq \frac12a|H|^2\varphi_R
\end{equation}
as $|y|\to\infty$ for any $s\in[0,S_0]$. 
On the other hand, we have for $J_r=J\cdot \frac{y}{|y|}$
\[\frac{J_r(y,s)}{|y|}=\frac{J(y,s)\cdot y}{|y|^2}=\frac{y}{|y|^2}\cdot \left(J(0,s)+ \int_0^1y\cdot\nabla J(ty, s)\, dt \right)\]
and 
\[\frac{|J_r(y,s)|}{|y|}\lesssim |\nabla J(y,s)|\]
as $|y|\to \infty$ for all $s\in[0,S_0]$. Hence it follows from the assumption \eqref{decay4} that 
\[J_r(y)=o(|y|) \ \ \ \mbox{as} \ \ |y|\to \infty.\]
Thus we deduce
\begin{equation}\label{far-2}
|(J\cdot\nabla)\varphi_R|=|J_r\partial_r\varphi_R|\leq -\frac{b}2|y|\partial_r\varphi_R=- \frac{b}2y\cdot\nabla \varphi_R
\end{equation}
as $|y|\to\infty$ for any $s\in[0,S_0]$, where we used $\alpha<-2$ and hence $b=\frac{1}{\alpha+2}<0$.

In the end, noting $\varphi_R$ is radially non-decreasing, we have
\[y\cdot\nabla \varphi_R=y_r|y|\partial_r\varphi_R=|y|\partial_r\varphi_R\geq 0,\]
which implies
\[b(y\cdot\nabla)\varphi_R|H|^2\leq 0.\]

Combining the analysis above with equation \eqref{energy-41} we obtain
\begin{equation}\label{energy-42}
\partial_s(|H|^2\varphi_R)+2a|H|^2\varphi_R+b(y\cdot\nabla)(|H|^2\varphi_R)-(J\cdot\nabla)(|H|^2\varphi_R)\leq 0
\end{equation}
for $R\geq R_0$ and $s\in[0,S_0]$, with $R_0>0$ large enough so that both \eqref{far-1} and \eqref{far-2} are satisfied. 

Now we suppose that $H$ is not compactly supported on $\mathbb R^3\times [0,S_0]$. Thus for any large $R>0$ we can find an open set $\Omega\subset (\mathbb R^3\backslash B_R)\times (0,S_0)$ such that $|H|(\Omega)>0$. Denote the domain 
\[D_{R_0}=(\mathbb R^3\backslash B_{R_0})\times (0,S_0)\]
and its boundary $\partial D_{R_0}=\Gamma_1\cup \Gamma_2\cup \Gamma_3\cup \Gamma_4$ with 
\begin{equation}\notag
\begin{split}
\Gamma_1&=\partial B_{R_0}\times (0,S_0), \ \quad \quad \Gamma_2=\{|y|=\infty\}\times (0,S_0),\\
\Gamma_3&=(\mathbb R^3\backslash B_{R_0})\times \{s=0\}, \ \quad \Gamma_4=(\mathbb R^3\backslash B_{R_0})\times \{s=S_0\}.
\end{split}
\end{equation}
Since $H$ is positive in $\Omega\subset D_{R_0}$, $|H|^2\varphi_{R_0}>0$ on $\Omega$ in view of the definition of $\varphi_{R_0}$. Thus there exists a positive maximum of $|H|^2\varphi_{R_0}$ on $\overline{D}_{R_0}$, denoted by $(y^*, s^*)$. On the other hand, $|H|^2\varphi_{R_0}$ satisfies the inequality \eqref{energy-42} with $a>0$ since $\alpha<-2$. 

We first claim that $(y^*, s^*)\not\in D_{R_0}\cup \Gamma_4$. Otherwise, we have at $(y^*, s^*)$
\[(y\cdot\nabla)(|H|^2\varphi_{R_0})(y^*, s^*)=0,  \ \quad (J\cdot\nabla)(|H|^2\varphi_{R_0})(y^*, s^*)=0, \ \quad (|H|^2\varphi_{R_0})(y^*,s^*)>0\]
and 
\[\lim_{s\to s^*}\partial_s(|H|^2\varphi_{R_0})(y^*, s)\geq 0.\]
This is in contradiction with \eqref{energy-42}.

Since $\varphi_{R_0}=0$ on $\partial B_{R_0}$, we have $|H|^2\varphi_{R_0}=0$ on $\Gamma_1$. Thus $(y^*, s^*)\not\in \Gamma_1$. In view of the assumption \eqref{decay4}, it is clear that $|H|^2\varphi_{R_0}=0$ on $\Gamma_2$ and hence $(y^*, s^*)\not\in \Gamma_2$.

Therefore we must have $(y^*, s^*)\in \Gamma_3$. In this case, we have
\begin{equation}\label{max1}
\sup_{|y|>R_0, 0<s<S_0} |H(y,s)|^2\varphi_{R_0}(y) \leq \sup_{|y|>R_0}|H(y,0)|^2\varphi_{R_0}(y)=|H(y^*,0)|^2\varphi_{R_0}(y^*).
\end{equation}
On the other hand, we have for all $(y,s)\in D_{R_0}\cup \Gamma_4$ 
\begin{equation}\label{max2}
 |H(y,s)|^2\varphi_{R_0}(y) < |H(y^*,0)|^2\varphi_{R_0}(y^*)
\end{equation}
where the strict inequality is due to the fact that $(y^*, s^*)\not\in D_{R_0}\cup \Gamma_4$. Taking $y=y^*$ and $s=S_0$ in \eqref{max2} yields 
\[ |H(y^*, S_0)|^2\varphi_{R_0}(y^*) < |H(y^*,0)|^2\varphi_{R_0}(y^*)\]
which is a contradiction to the periodicity $H(y^*, S_0)=H(y^*, 0)$. Therefore, our assumption that $H$ is not compactly supported is false. 

As a consequence of $H$ being compactly supported, the assumption (iii) of Theorem \ref{thm-1} is naturally satisfied. 

Next we claim that the assumption \eqref{decay4} implies 
\[J(y)=o(|y|) \ \ \ \mbox{as} \ \ |y|\to \infty,\]
that is, the assumption (ii) of Theorem \ref{thm-1} is also satisfied.
This is trivial in 1D. In higher dimension, we denote $\hat y=\frac{y}{|y|}$ and $f(\tau)=J(\tau\hat y)$ for $\tau\in \mathbb R$. It follows
\[f'(\tau)=\nabla J(t\hat y)\cdot \hat y.\]
Since $\nabla J(y)\to 0$ as $|y|\to \infty$, we see that
\[f'(\tau)\to 0 \ \ \ \mbox{as}\ \ \tau\to\infty.\]
Thus 
\[f(\tau)=o(\tau) \ \ \ \mbox{as} \ \ \tau\to \infty,\]
which immediately indicates 
\[J(y)=o(|y|) \ \ \ \mbox{as} \ \ |y|\to \infty.\]

In summary, we showed that under the condition $\alpha<-2$ and \eqref{decay4}, the assumptions of Theorem \ref{thm-1} are satisfied and hence $H\equiv 0$ on $\mathbb R^{3+1}$.

\bigskip

\section{Proof of Theorem \ref{thm-5}}
\label{sec-thm5}

Take $\varphi_R$ to be the smooth radial cut-off function defined in Section \ref{sec-thm1} that is non-increasing. Assume $y=0$ is an extremal point of $J_i(y,s)$ with $i=1,2,3$ for all $s\in(0,S_0)$ on the ball $B_{R_0}$ for some $R_0>0$. 
Denote
\[M=\sup_{0<s<S_0}\|D^2J\|_{L^\infty(B_R)}\]
for sufficiently large $R>R_0$.

The goal is to show that $H\equiv 0$ on $B_{r_k}$ inductively for an increasing sequence $\{r_k\}_{k=0}^{k=N}$ with the last one $r_N\approx R$. Since $R$ can be arbitrarily large, it jutisfies the conclusion of the theorem.

The first step is to find $r_0\leq R_0$ depending on $R_0, M$ and $\alpha$ such that $H\equiv 0$ on $B_{r_0}$. For $y\in B_{r_0}$, we have 
\[ J(0,s)=D J(0,s)=0, \ \ \forall \ \ s\in(0,S_0).\]
It follows that
\begin{equation}\label{est-J1}
|D J(y,s)|=|D J(y,s)-DJ(0,s)|\leq \sup_{0<s<S_0}\|D^2J(s)\|_{L^\infty(B_R)}|y|\leq r_0M
\end{equation}
for $y\in B_{r_0}$, and 
\begin{equation}\label{est-J2}
\begin{split}
|J(y,s)|&=|J(0,s)-\int_0^1y\cdot \nabla J(\tau y,s)\, d\tau|\\
&\leq |y| \sup_{|z|\leq |y|}|DJ(z,s)|\\
&\leq |y|r_0M
\end{split}
\end{equation}
for any $(y,s)\in B_{r_0}\times (0,S_0)$.

Taking dot product of the profile equation \eqref{eq-H2} with $H\varphi_{r_0}$ yields
\begin{equation}\label{energy-51}
\begin{split}
&\quad\partial_s(|H|^2\varphi_{r_0})+2a|H|^2\varphi_{r_0}+b(y\cdot\nabla)(|H|^2\varphi_{r_0})-(J\cdot\nabla)(|H|^2\varphi_{r_0})\\
&=-2(H\cdot\nabla)J\cdot H\varphi_{r_0}+b(y\cdot\nabla \varphi_{r_0}) |H|^2-(J\cdot\nabla \varphi_{r_0}) |H|^2.
\end{split}
\end{equation}
The right hand side of \eqref{energy-51} can be bounded as follows, using \eqref{est-J1} and \eqref{est-J2}
\begin{equation}\label{energy-52}
\begin{split}
&\quad -2(H\cdot\nabla)J\cdot H\varphi_{r_0}+b(y\cdot\nabla \varphi_{r_0}) |H|^2-(J\cdot\nabla \varphi_{r_0}) |H|^2\\
&\leq 2|\nabla J| |H|^2\varphi_{r_0}+b(y\cdot\nabla \varphi_{r_0}) |H|^2- J_r\partial_r \varphi_{r_0} |H|^2\\
&\leq 2r_0M |H|^2\varphi_{r_0}+b(y\cdot\nabla \varphi_{r_0}) |H|^2- |y|\partial_r \varphi_{r_0} |H|^2r_0M\\
&= 2r_0M |H|^2\varphi_{r_0}+b(y\cdot\nabla \varphi_{r_0}) |H|^2- (y\cdot\nabla \varphi_{r_0}) |H|^2r_0M\\
&\leq a |H|^2\varphi_{r_0}+\frac12b(y\cdot\nabla \varphi_{r_0}) |H|^2
\end{split}
\end{equation}
provided $2r_0M\leq a$ and $r_0M\leq \frac12b$. Thus, for $\alpha>0$, taking 
\[r_0=\min\left\{R_0, \frac{\alpha}{2M(\alpha+2)},\frac{1}{2M(\alpha+2)} \right \}\]
and combining \eqref{energy-51}-\eqref{energy-52}, we deduce
\begin{equation}\label{energy-53}
\begin{split}
&\quad\partial_s(|H|^2\varphi_{r_0})+a|H|^2\varphi_{r_0}+b(y\cdot\nabla)(|H|^2\varphi_{r_0})-(J\cdot\nabla)(|H|^2\varphi_{r_0})\\
&\leq \frac12b(y\cdot\nabla \varphi_{r_0}) |H|^2.
\end{split}
\end{equation}
Since $\varphi_{r_0}$ is non-increasing, we have
\[b(y\cdot\nabla \varphi_{r_0})=b|y|\partial_r \varphi_{r_0}\leq 0.\]
Thus it follows from \eqref{energy-53} that
\begin{equation}\label{energy-54}
\partial_s\Phi+a\Phi+b(y\cdot\nabla)\Phi-(J\cdot\nabla)\Phi\leq 0
\end{equation}
with $\Phi (y,s):=|H(y,s)|^2\varphi_{r_0}(y)$. For $q>2$, multiplying \eqref{energy-54} by $|\Phi|^{q-2}\Phi$ and integrating over $B_{r_0}\times [0,S_0]$ we obtain
\begin{equation}\label{energy-55}
\begin{split}
&\quad \frac1q \int_0^{S_0}\int_{B_{r_0}}\partial_s |\Phi|^q\, dyds+a\int_0^{S_0}\int_{B_{r_0}} |\Phi|^q\, dyds\\
&+\frac{b}{q}\int_0^{S_0}\int_{B_{r_0}} (y\cdot\nabla)|\Phi|^q\, dyds-\frac{1}{q}\int_0^{S_0}\int_{B_{r_0}} (J\cdot\nabla)|\Phi|^q\, dyds\leq 0.
\end{split}
\end{equation}
Applying periodicity of $H$ in $s$, it is clear that 
\[\int_0^{S_0}\int_{B_{r_0}}\partial_s |\Phi|^q\, dyds=\int_{B_{r_0}} |\Phi|^q(y,S_0)\, dy-\int_{B_{r_0}} |\Phi|^q(y,0)\, dy=0.\]
Using integration by parts, we have
\begin{equation}\notag
\begin{split}
\int_0^{S_0}\int_{B_{r_0}} (y\cdot\nabla)|\Phi|^q\, dyds&=\int_0^{S_0}\int_{B_{r_0}}\div (y|\Phi|^q)\, dyds-\int_0^{S_0}\int_{B_{r_0}}(\div y)|\Phi|^q\, dyds\\
&=-3\int_0^{S_0}\int_{B_{r_0}}|\Phi|^q\, dyds,
\end{split}
\end{equation}
and 
\begin{equation}\notag
\begin{split}
\int_0^{S_0}\int_{B_{r_0}} (J\cdot\nabla)|\Phi|^q\, dyds&=\int_0^{S_0}\int_{B_{r_0}}\div (J|\Phi|^q)\, dyds-\int_0^{S_0}\int_{B_{r_0}}(\div J)|\Phi|^q\, dyds\\
&=0.
\end{split}
\end{equation}
Combining the last three equations and the energy inequality \eqref{energy-55} yields
\begin{equation}\notag
(a-\frac{3b}q)\int_0^{S_0}\int_{B_{r_0}} |\Phi|^q\, dyds\leq 0.
\end{equation}
Choosing $q>\frac{3b}{a}$, the last inequality implies $\Phi\equiv 0$ and hence $H\equiv 0$ on $B_{r_0}$.

Next we define the increasing sequence $\{r_k\}$ iteratively
\[r_k=r_{k-1}+m, \ \ \ k=1, 2, ..., N\]
with $m=\min\{\frac{\alpha}{2M(\alpha+2)}, \frac{1}{2M(\alpha+2)}\}$, and $N$ being the smallest integer such that $r_N=r_0+Nm\geq R$. We assume that 
$H\equiv 0$ on $B_{r_{k-1}}\times (0,S_0)$, and aim to show that $H\equiv 0$ on $B_{r_{k}}\times (0,S_0)$. 

For $y\in B_{r_k}$, denote $\bar y=\frac{y}{|y|}r_{k-1}\in\partial B_{r_{k-1}}$. Since $H\equiv 0$ on $B_{r_{k-1}}\times (0,S_0)$, we have
\begin{equation}\label{est-J1-k}
|D J(y,s)|\leq |y-\bar y|\sup_{0<s<S_0}\|D^2J(s)\|_{L^\infty(B_R)}\leq (r_k-r_{k-1})M\leq mM
\end{equation}
for $y\in B_{r_k}$, and 
\begin{equation}\label{est-J2-k}
\begin{split}
|J(y,s)|&=|J(\bar y,s)-\int_0^1(y-\bar y)\cdot \nabla J(\tau (y-\bar y)+\bar y,s)\, d\tau|\\
&\leq |y-\bar y| \sup_{|z|\leq |y|}|DJ(z,s)|\\
&\leq |y-\bar y|mM\\
&\leq |y|mM
\end{split}
\end{equation}
for any $(y,s)\in B_{r_k}\times (0,S_0)$. Now taking dot product of \eqref{eq-H2} with $H\varphi_{r_k}$ and applying \eqref{est-J1-k}-\eqref{est-J2-k} gives
\begin{equation}\label{energy-57}
\begin{split}
&\quad\partial_s(|H|^2\varphi_{r_k})+2a|H|^2\varphi_{r_k}+b(y\cdot\nabla)(|H|^2\varphi_{r_k})-(J\cdot\nabla)(|H|^2\varphi_{r_k})\\
&=-2(H\cdot\nabla)J\cdot H\varphi_{r_k}+b(y\cdot\nabla \varphi_{r_k}) |H|^2-(J\cdot\nabla \varphi_{r_k}) |H|^2\\
&\leq 2|\nabla J| |H|^2\varphi_{r_k}+b(y\cdot\nabla \varphi_{r_k}) |H|^2- J_r\partial_r \varphi_{r_k} |H|^2\\
&\leq 2mM |H|^2\varphi_{r_k}+b(y\cdot\nabla \varphi_{r_k}) |H|^2- (y\cdot\nabla \varphi_{r_k}) |H|^2mM\\
&\leq a |H|^2\varphi_{r_k}+\frac12b(y\cdot\nabla \varphi_{r_k}) |H|^2
\end{split}
\end{equation}
where in the last step we used  $2mM\leq a$ and $mM\leq \frac12b$ thanks to the choice of $m$. Again, since $b(y\cdot\nabla \varphi_{r_k})\leq 0$, it follows from \eqref{energy-57} that
\[\partial_s(|H|^2\varphi_{r_k})+a|H|^2\varphi_{r_k}+b(y\cdot\nabla)(|H|^2\varphi_{r_k})-(J\cdot\nabla)(|H|^2\varphi_{r_k})\leq 0.\]
Therefore, following a similar argument as on $B_{r_0}$, we conclude that $H\equiv 0$ on $B_{r_k}$. By induction, we can show that $H\equiv 0$ on $B_{r_N}$ with $r_N\geq R$. Since $R$ can be arbitrarily large, we have $H\equiv 0$ on $\mathbb R^{3+1}$.

\bigskip

\section{Proof of Theorem \ref{thm-6}}
\label{sec-thm6}

We need the following iterative lemma proved in \cite{CW-23}. 

\begin{Lemma}\label{le-iterative}
For $k\in \mathbb N\cup \{0\}$, let $\beta_k: [a,b]\to \mathbb R$ be continuous functions. Assume that there exists a constant $M>0$ such that
\[\sup_{t\in [a,b]} |\beta_k(t)|\leq M^k \ \ \forall \ k\in \mathbb N\cup \{0\}.\]
Let $\alpha:[a,b]\to\mathbb R$ be a nonnegative function satisfying $\alpha\in L^1(a,b)$. Moreover, assume that there exists a constant $C>0$ such that 
\begin{equation}\notag
\beta_k(t)\leq C+\int_a^t \alpha(s)\beta_{k+1}(s)\, ds, \ \ \forall \k\in \mathbb N\cup \{0\}.
\end{equation}
Then the estimate 
\[\beta_0(t)\leq C e^{\int_a^t \alpha(s)\, ds}\]
holds for all $t\in[a,b]$.
\end{Lemma}

Applying Lemma \ref{le-iterative}, we establish the following local energy estimate for the electron MHD which may be interesting by itself. 

\begin{Proposition}\label{prop-energy}
Let $a,b\in \mathbb R$ with $a<b$, and $B\in C^1([a,b), C^2(\mathbb R^3))$ be a solution to the non-resistive electron MHD \eqref{emhd} with $\nu=0$. Assume 
\begin{equation}\label{ass-61}
\sup_{s\in[a,t]}\|(1+|x|)^{-2}B(s)\|_{L^\infty(\mathbb R^3)}<\infty, \ \ \forall \ t\in[a,b), 
\end{equation}
\begin{equation}\label{ass-62}
\int_a^b\|(1+|x|)^{-1}\nabla\times B(s)\|_{L^\infty(\mathbb R^3)}\, ds <\infty.
\end{equation}
Then we have
\begin{equation}\label{energy-local1}
r^{-7}\|B(t)\|_{L^2(B_{r})}^2\leq c \|(r+|x|)^{-2}B(a)\|_{L^\infty(\mathbb R^3)}^2e^{c\int_a^t \|(r+|x|)^{-1}\nabla\times B(s)\|_{L^\infty(\mathbb R^3)}\, ds}
\end{equation}
for all $t\in(a,b)$ and all $r>0$. 
\end{Proposition}
\begin{proof}
Denote the ball $\Omega=B_R(x_0)$ for some fixed $x_0$ and $R$. Let $a,b\in \mathbb R$ and $a<b$. For $\varphi \in C_c^\infty (\Omega)$ and $t\in(a,b)$, we have 
\begin{equation}\notag
\int_a^t\int_{\Omega} \partial_t B\cdot B\varphi^2 dxds+\int_a^t\int_{\Omega} \nabla\times ((\nabla\times B)\times B)\cdot B\varphi^2 dxds=0.
\end{equation}
Note 
\begin{equation}\notag
\int_a^t\int_{\Omega} \partial_t B\cdot B\varphi^2 dxds=\frac12\int_{\Omega} |B(t)|^2\varphi^2 dx-\frac12\int_{\Omega} |B(a)|^2\varphi^2 dx,
\end{equation}
\begin{equation}\notag
\begin{split}
&\quad \int_a^t\int_{\Omega} \nabla\times ((\nabla\times B)\times B)\cdot B\varphi^2 dxds\\
&=\int_a^t\int_{\Omega} ((\nabla\times B)\times B)\cdot \nabla\times (B\varphi^2) dxds\\
&=\int_a^t\int_{\Omega} ((\nabla\times B)\times B)\cdot (\nabla\times B)\varphi^2 dxds-2\int_a^t\int_{\Omega} ((\nabla\times B)\times B)\cdot (B\times \nabla\varphi)\varphi dxds\\
&=-2\int_a^t\int_{\Omega} ((\nabla\times B)\times B)\cdot (B\times \nabla\varphi)\varphi dxds,
\end{split}
\end{equation}
hence we have
\begin{equation}\label{energy-61}
\int_{\Omega} |B(t)|^2\varphi^2 dx=\int_{\Omega} |B(a)|^2\varphi^2 dx+4\int_a^t\int_{\Omega} ((\nabla\times B)\times B)\cdot (B\times \nabla\varphi)\varphi dxds.
\end{equation}

Denote $r_k=2^k$ for $k\in \mathbb N\cup \{0\}$ and $\rho_k=\frac{2r_k+2r_{k-1}}{4}=\frac34r_k$. Clearly we have $r_{k-1}<\rho_k<r_k$. Choose the smooth cut-off function $\varphi_k\in C_c^\infty (B_{\rho_k})$ satisfying 
\[0\leq \varphi_k\leq 1 \ \ \mbox{on} \ B_{\rho_k}; \ \ \varphi_k\equiv 1 \ \ \mbox{on} \ \ B_{r_{k-1}}.\]
It is easy to see that 
\[|\nabla\varphi_x|\leq c 2^{-k} \ \ \mbox{in} \ \ B_{\rho_k}.\]

Taking $\Omega=B_{r_k}$ and $\varphi=\varphi_k$ in \eqref{energy-61} and noting
\begin{equation}\notag
\begin{split}
&\quad\int_a^t\int_{B_{r_k}} |((\nabla\times B)\times B)\cdot (B\times \nabla\varphi)\varphi| dxds\\
&\leq c2^{-k}\int_a^t\int_{B_{r_k}} |\nabla \times B| |B|^2 dxds\\
&\leq c2^{-k}\int_a^t\|\nabla\times B\|_{L^\infty(B_{r_k})} \|B\|_{L^2(B_{r_k})}^2ds,
\end{split}
\end{equation}
we have 
\begin{equation}\label{energy-62}
\begin{split}
&\quad \|B(t)\|_{L^2(B_{r_{k-1}})}^2\\
&\leq c \|B(a)\|_{L^2(B_{r_{k}})}^2+c2^{-k}\int_a^t\|\nabla\times B\|_{L^\infty(B_{r_k})} \|B\|_{L^2(B_{r_k})}^2ds\\
&\leq c \|B(a)\|_{L^2(B_{r_{k}})}^2+c\int_a^t\|(1+|x|)^{-1}\nabla\times B\|_{L^\infty(\mathbb R^3)} \|B\|_{L^2(B_{r_k})}^2ds.
\end{split}
\end{equation}

Denote $\beta_k(t)=2^{-7k}\|B(t)\|_{L^2(B_{r_k})}^2$ and $\alpha(t)=\|(1+|x|)^{-1}\nabla\times B(t)\|_{L^\infty(\mathbb R^3)}$. It follows from \eqref{energy-62} that
\[ 2^{7(k-1)}\beta_{k-1}(t)\leq c 2^{7k}\beta_k(a)+c\int_a^t 2^{7k}\beta_k(s)\alpha(s)\, ds\]
and hence
\begin{equation}\label{energy-63}
\beta_{k-1}(t)\leq 128c\beta_k(a)+128c\int_a^t \beta_k(s)\alpha(s)\, ds.
\end{equation}
On the other hand, we have
\begin{equation}\notag
\begin{split}
\beta_k(t)&=2^{-7k}\int_{B_{r_k}} |B(x,t)|^2\, dx\\
&\leq 2^{-7k+3k} \|B(t)\|_{L^\infty(B_{r_k})}^2\\
&\leq c \|(1+|x|)^{-2}B(t)\|_{L^\infty(\mathbb R^3)}^2.
\end{split}
\end{equation}
In view of the assumption 
\[\sup_{s\in[a,t]}\|(1+|x|)^{-2}B(s)\|_{L^\infty(\mathbb R^3)}<\infty, \ \ \forall \ t\in[a,b), \]
there exists some function $C(t)>0$ such that
\[\beta_k(s)\leq C(t), \ \ \forall \ k\geq 0, \ \ \forall \ t\in[a,b). \]
Thus applying Lemma \ref{le-iterative} to \eqref{energy-63} gives
\begin{equation}\label{energy-64}
\beta_0(t)\leq 128c \|(1+|x|)^{-2}B(a)\|_{L^\infty(\mathbb R^3)}^2e^{\int_a^t\alpha(s)\, ds}.
\end{equation}
Recall the scaling $B_\lambda(x,t)=\lambda^{\alpha}B(\lambda x,\lambda^{\alpha+2}t)$ and in particular
$B_\lambda(x,t)=\lambda^{-2}B(\lambda x,t)$ for $\alpha=-2$. Thus, $\tilde B(x,t)=: r^{-2}B(rx,t)$ is also a solution to the electron MHD and satisfies the energy inequality \eqref{energy-64}.
Therefore plugging $\tilde B$ in the energy inequality \eqref{energy-64} and rescaling it gives the local energy inequality \eqref{energy-local1}.

\end{proof}

We are now ready to prove Theorem \ref{thm-6}. For any $\epsilon>0$ (fixed), the assumption \eqref{non-decay} implies that there exists $y_0>1$ such that
\[|H(y,s)|\leq \epsilon |y|^2 \ \ \forall \ \ (y,s)\in\mathbb R^3\times \mathbb R \ \ \mbox{with} \ |y|\geq y_0.\]
Recalling \eqref{self-blowup}, we have for $|x|\geq y_0(-t)^{\frac1{\alpha+2}}$, 
\begin{equation}\notag
\begin{split}
(1+|x|)^{-2}|B(x,t)|&\leq (1+|x|)^{-2}(-t)^{-\frac{\alpha}{\alpha+2}}\epsilon |y|^2\\
&\leq (1+|x|)^{-2}(-t)^{-\frac{\alpha}{\alpha+2}}\epsilon (-t)^{-\frac2{\alpha+2}}|x|^2\\
&\leq \epsilon (-t)^{-1}.
\end{split}
\end{equation}
While for $|x|< y_0(-t)^{\frac1{\alpha+2}}$, it follows from the assumption \eqref{non-decay} that 
\[\sup_{s\in\mathbb R} |H(y,s)|\leq c (1+|y|)^2\]
for some constant $c_0>0$, and hence
\begin{equation}\notag
\begin{split}
(1+|x|)^{-2}|B(x,t)|&\leq (1+|x|)^{-2}(-t)^{-\frac{\alpha}{\alpha+2}}c_0 (1+|y|)^2\\
&\leq (1+|x|)^{-2}(-t)^{-\frac{\alpha}{\alpha+2}}c_0 (1+ (-t)^{-\frac{1}{\alpha+2}}|x|)^2\\
&\leq (1+|x|)^{-2}(-t)^{-\frac{\alpha}{\alpha+2}}c_0 (1+ y_0)^2\\
&\leq c_0 (1+ y_0)^2 (-t)^{-\frac{2}{\alpha+2}}(-t)^{-1}\\
&\leq \epsilon (-t)^{-1}
\end{split}
\end{equation}
where the last step is obtained by taking $a_1\in (-1,0)$ such that 
\[c_0 (1+ y_0)^2 (-a_1)^{-\frac{2}{\alpha+2}}\leq \epsilon.\]
Therefore we have
\[\| (1+|x|)^{-2}B(t)\|_{L^\infty(\mathbb R^3)}\leq \epsilon (-t)^{-1} \ \ \forall \ t\in[a_1,0),\]
and hence the assumption \eqref{ass-61} is satisfied for $t\in[a_1,0)$.

On the other hand, it follows from the assumption \eqref{non-decay} that there exists $y_0>1$ such that
\[|\nabla\times H(y,s)|\leq \epsilon |y| \ \ \forall \ \ (y,s)\in\mathbb R^3\times \mathbb R \ \ \mbox{with} \ |y|\geq y_0.\]
Similar analysis as above shows that 
\[\| (1+|x|)^{-1}\nabla\times B(t)\|_{L^\infty(\mathbb R^3)}\leq \epsilon (-t)^{-1} \ \ \forall \ t\in[a_2,0),\]
for some $a_2\in (-1,0)$. Take $a=\max\{a_1,a_2\}$. Thus both assumptions \eqref{ass-61} and \eqref{ass-62} are satisfied with such $a$ and $b=0$. It then follows from \eqref{energy-local1} with $r=1$ that
\begin{equation}\label{energy-local2}
\begin{split}
\|B(\tau)\|_{L^2(B_1)}^2&\leq c\|(1+|x|)^{-2}B(t)\|_{L^\infty(\mathbb R^3)}^2 e^{c\int_t^\tau \|(1+|x|)^{-1}\nabla\times B(s)\|_{L^\infty(\mathbb R^3)}\, ds}\\
&\leq c \epsilon (-t)^{-2} e^{c\epsilon(-\log (-\tau)+\log(-t))}\\
&\leq c \epsilon (-t)^{-2} e^{c\epsilon(-\log (-\tau))}\\
&\leq c \epsilon (-t)^{-2} (-\tau)^{-c\epsilon}.
\end{split}
\end{equation}

For any $k\in \mathbb N$, we have
\[B(x,t)=\lambda^{k\alpha}B(\lambda^kx, \lambda^{k(\alpha+2)}t)\]
which indicates after rescaling in time
\[B(x, \lambda^{-k(\alpha+2)}t)=\lambda^{k\alpha}B(\lambda^kx, t)\]
and hence 
\[B(\lambda^kx, t)=\lambda^{-k\alpha}B(x, \lambda^{-k(\alpha+2)}t).\]
It then follows that
\begin{equation}\label{energy-local3}
\begin{split}
\|B(t)\|_{L^2(B_{\lambda^k})}^2&=\int_{B_{\lambda^k}}B^2(x,t)\, dx=\lambda^{3k}\int_{B_1}B^2(\lambda^kz,t)\, dz\\
&=\lambda^{3k-2k\alpha}\int_{B_1}B^2(z,\lambda^{-k(\alpha+2)}t)\, dz.
\end{split}
\end{equation}
Applying \eqref{energy-local2} with $\tau=\lambda^{-k(\alpha+2)}t$ to \eqref{energy-local3} yields
\begin{equation}\label{energy-local4}
\begin{split}
\|B(t)\|_{L^2(B_{\lambda^k})}^2&\leq \lambda^{3k-2k\alpha} c\epsilon(-t)^{-2}(-\lambda^{-k(\alpha+2)}t)^{-c\epsilon}\\
&=c\epsilon \lambda^{3k-2k\alpha+c\epsilon k(\alpha+2)}(-t)^{-2-c\epsilon}.
\end{split}
\end{equation}
Since $\alpha>\frac32$, we can choose $\epsilon>0$ small enough (independent of $k$) such that we have $3-2\alpha+c\epsilon (\alpha+2)<0$. Therefore, taking $k\in \infty$, we have from \eqref{energy-local4}
\[\|B(t)\|_{L^2(B_{\lambda^k})}^2 \to 0.\]
By translation in time, we have $B\equiv 0$ and hence $H\equiv 0$ on $\mathbb R^3\times \mathbb R$.







\end{document}